\documentclass[12pt]{amsart}

\title[Diameters of hyperbolic surfaces]{A sharper bound on the minimal possible diameter of a closed hyperbolic surface}

\author{Joffrey Mathien}
\address[Joffrey Mathien]{Universität Wien, Fakultät für Mathematik. Oskar Morgenstern Platz 1, Wien 1090, Austria}
\email{joffrey.mathien@univie.ac.at}

\author{Bram Petri}
\address[Bram Petri]{Institut de Math\'ematiques de Jussieu--Paris Rive Gauche and  Institut universitaire de France ; Sorbonne Universit\'e and Universit\'e Paris Cit\'e, CNRS, IMJ-PRG, F-75005 Paris, France}
\email{bram.petri@imj-prg.fr}

\date{\today}

\usepackage{amsmath,amsfonts,amssymb,amsthm,graphicx,overpic,mathtools}
\usepackage{float,enumitem}
\usepackage{color,xcolor}
\usepackage[colorinlistoftodos]{todonotes}
\usepackage{multirow}
\usepackage{pgfplots}
\usepackage{subcaption}
\usepackage[colorlinks=true,linkcolor=red,citecolor=blue]{hyperref}

\pgfplotsset{compat=1.7}

\usepackage[margin=2.7cm]{geometry}  
\usepackage[indent]{parskip}

\numberwithin{equation}{section}

\newtheorem{thm}{Theorem}[section]
\newtheorem{prop}[thm]{Proposition}
\newtheorem{prp}[thm]{Proposition}
\newtheorem{cor}[thm]{Corollary}
\newtheorem{lem}[thm]{Lemma}

\theoremstyle{definition}

\newtheorem{defn}[thm]{Definition}

\newcommand{\nc}{\newcommand}
\nc{\dmo}{\DeclareMathOperator}
\usepackage{dsfont}

\nc{\abs}[1]{\left| #1 \right|}
\nc{\bigO}[1]{O\left(#1\right)}
\nc{\card}[1]{\left|#1\right|}
\nc{\ceil}[1]{\left\lceil #1 \right\rceil}
\nc{\CC}{\mathbb{C}}
\nc{\dilog}{\mathcal{L}}
\nc{\floor}[1]{\left\lfloor #1 \right\rfloor}
\nc{\ind}{\mathds{1}}
\nc{\ZZ}{\mathbb{Z}}
\nc{\len}[1]{\left| #1 \right|}
\nc{\littleo}[1]{o\left(#1\right)}
\dmo{\Mat}{Mat}
\nc{\NN}{\mathbb{N}}
\nc{\norm}[1]{\left|\left| #1 \right|\right|}
\nc{\QQ}{\mathbb{Q}}
\nc{\RR}{\mathbb{R}}
\nc{\st}[2]{\left\{\, #1 \,:\, #2\,\right\}}
\dmo{\supp}{supp}
\nc{\tr}[1]{\mathrm{tr}\left(#1\right)}
\nc{\what}{\widehat}
\dmo{\im}{Im}
\nc{\eps}{\varepsilon}
\dmo{\li}{li}
\dmo{\arccosh}{arccosh}

\dmo{\area}{area}
\dmo{\conv}{conv}
\dmo{\diam}{diam}
\dmo{\DD}{\mathbb{D}}
\dmo{\dist}{\mathrm{d}}
\nc{\HH}{\mathbb{H}}
\dmo{\Isom}{Isom}
\dmo{\MCG}{MCG}
\dmo{\MPL}{MPL}
\dmo{\Mod}{\mathcal{M}}
\dmo{\PL}{PL}
\nc{\Sphere}{\mathbb{S}}
\dmo{\sys}{sys}
\dmo{\kiss}{Kiss}
\dmo{\Teich}{\mathcal{T}}
\nc{\Torus}{\mathbb{T}}
\dmo{\vol}{vol}
\dmo{\WP}{WP}

\dmo{\convTV}{\;\stackrel{\mathrm{TV}}{\longrightarrow}\;}
\nc{\ExV}[2]{\mathbb{E}_{#1}\left[#2\right]}
\dmo{\EE}{\mathbb{E}}
\nc{\Pro}[2]{\mathbb{P}_{#1}\left[#2\right]}
\dmo{\PP}{\mathbb{P}}
\nc{\distTV}[2]{\mathrm{d}_{\rm TV}\left(#1,#2\right)}
\dmo{\UU}{\mathbb{U}}
\nc{\Var}[2]{\mathbb{V}\mathrm{ar}_{#1}\left[#2\right]}

\dmo{\alt}{\mathfrak{A}}
\dmo{\Aut}{Aut}
\dmo{\Fix}{Fix}
\dmo{\GL}{GL}
\dmo{\Hom}{Hom}
\dmo{\id}{Id}
\dmo{\PGL}{PGL}
\dmo{\PSL}{PSL}
\dmo{\PO}{PO}
\dmo{\Rep}{Rep}
\dmo{\SL}{SL}
\dmo{\SO}{SO}
\dmo{\sym}{\mathfrak{S}}
\dmo{\inv}{\mathcal{I}}
\dmo{\orb}{\mathcal{O}}
\dmo{\stab}{Stab}

\dmo{\calA}{\mathcal{A}}
\dmo{\calB}{\mathcal{B}}
\dmo{\calC}{\mathcal{C}}
\dmo{\calD}{\mathcal{D}}
\dmo{\calE}{\mathcal{E}}
\dmo{\calF}{\mathcal{F}}
\dmo{\calG}{\mathcal{G}}
\dmo{\calH}{\mathcal{H}}
\dmo{\calI}{\mathcal{I}}
\dmo{\calJ}{\mathcal{J}}
\dmo{\calK}{\mathcal{K}}
\dmo{\calL}{\mathcal{L}}
\dmo{\calM}{\mathcal{M}}
\dmo{\calN}{\mathcal{N}}
\dmo{\calO}{\mathcal{O}}
\dmo{\calP}{\mathcal{P}}
\dmo{\calQ}{\mathcal{Q}}
\dmo{\calR}{\mathcal{R}}
\dmo{\calS}{\mathcal{S}}
\dmo{\calT}{\mathcal{T}}
\dmo{\calU}{\mathcal{U}}
\dmo{\calV}{\mathcal{V}}
\dmo{\calW}{\mathcal{W}}
\dmo{\calX}{\mathcal{X}}
\dmo{\calY}{\mathcal{Y}}
\dmo{\calZ}{\mathcal{Z}}

\nc{\klav}{Klav\v{z}ar}

\nc{\bi}{\mathbf{i}}
\nc{\bj}{\mathbf{j}}
\nc{\bk}{\mathbf{k}}

\begin{document}

\begin{abstract} We prove that the minimal possible diameter of a closed hyperbolic surface of genus $g$ is at most $\log(g)+25 \log \log(g) + O(1)$.
\end{abstract}

\maketitle

\section{Introduction}

The diameter is the simplest measure of connectivity of a closed hyperbolic surface, other examples being its spectral gap and its Cheeger constant. In \cite{BudzinskiCurienPetri}, Budzinski, Curien and the second named author proved that
\[
\min_{X\in\calM_g}\left\{\mathrm{diam}(X) \right\} \stackrel{g\to\infty}{=} \log(g) + o(\log(g)),
\]
where $\calM_g$ denotes the moduli space of closed and orientable hyperbolic surfaces of genus $g$. The lower bound in this estimate is classical, so in order to prove it, one needs to build surfaces of high genus with a very small diameter. This is done using a random construction based on randomly gluing hyperbolic pairs of pants together. This construction was generalized by the first named author in \cite{Mathien}.

The best known lower bound on the diameter of a hyperbolic surface of genus $g$ is due to Bavard \cite{Bavard} and states that for $X\in\calM_g$,
\[
\mathrm{diam}(X) \geq \mathrm{arccosh}\left(\frac{1}{\sqrt{3}\tan(\pi/(12g-6))}\right) \stackrel{g\to\infty}{=}  \log(g) + \log(8\sqrt{3}) + o(1) .
\]
In particular, there is a gap left in the second order term. In the case of regular graphs, that inspired the works above, Bollobás and Fernandez-de-la-Vega \cite{BollobasFernandezdelaVega} proved a sharper bound on the second order term. Namely, their bound is of the order $\log\log$ of the size of the graph. The main reason that this doesn't automatically apply to the surface setting is geometric: estimates on the area growth in certain hyperbolic surfaces need to be made more uniform than those applied in \cite{BudzinskiCurienPetri}.

The goal of this note is to work this out and confirm the suspicion that was pronounced in \cite[Section 4]{BudzinskiCurienPetri}, using an idea from \cite{Mathien}. The result we obtain is:
\begin{thm}\label{th:main} There exists a universal constant $C>0$ such that for all $g\geq 2$,
\[
\min_{X\in\calM_g}\left\{\mathrm{diam}(X) \right\} \leq \log(g) + 25\cdot \log\log(g) + C
\]
\end{thm}

We can also think of this theorem in terms of covering the surface, in an analogous way to how the thickness (or covering density) of a lattice in $\mathbb{R}^n$ relates to the diameter of the corresponding flat torus. If $X$ is a closed hyperbolic surface and $x\in X$, the covering radius $\mathrm{rad}_X(x)$ at $x$ is the smallest radius such that the disk around $x$ of that radius covers $X$. Equivalently, this is the minimal $r>0$ such that the universal covering map is surjective when restricted to the closed disk $B_r(\widetilde{x})$ of radius $r$ around a lift $\widetilde{x}\in\mathbb{H}^2$ of $x$. The lack of injectivity on this disk can be quantified using the thickness
\[
\theta(x) = \frac{\mathrm{area}(B_{\mathrm{rad}_X(x)}(\widetilde{x}))}{\mathrm{area}(X)}. 
\]
at $x$. In analogy with the theory of lattice packings, we call the average\footnote{In the case of flat tori, one doesn't take the average, because the covering radius is constant.} of this quantity, i.e. $\Theta(X) := \frac{1}{\mathrm{area}(X)}\int_X \theta(x)\;d\mu(x)$, where $\mu$ denotes the hyperbolic area measure, the \textbf{thickness} of $X$. Theorem \ref{th:main} implies the following corollary:
\begin{cor}
There exists a universal constant $C>0$ such that
\[
\min_{X\in\calM_g}\left\{\Theta(X) \right\} \leq C \cdot \log(g)^{25}
\]
for all $g \geq 2$.
\end{cor}

\begin{proof}
This is immediate from the fact that $\mathrm{rad}_X(x)\leq \mathrm{diam}(X)$ for all $x\in X$.
\end{proof}

Like in the aforementioned papers, we use a random construction. In fact, we use the exact same random construction as in \cite{BudzinskiCurienPetri}. That is, for $g\geq 2$ and $\ell>0$, we define a random closed and orientable hyperbolic surface $S_{g,\ell}$ of genus $g$. To describe $S_{g,\ell}$, we will write $P_\ell$ for a hyperbolic pair of pants with three boundary components of length $\ell$. The surface $S_{g,\ell}$ is now obtained by sampling a random cubic graph on $2g-2$ vertices using the configuration model and then gluing $2g-2$ copies of $P_\ell$ together with twist parameter $0$ according to the combinatorics of this graph. 

The main technical result of this note is a sharper bound on $\mathrm{diam}(S_{g,\ell})$ than the one that was derived in \cite{BudzinskiCurienPetri}. We prove that, if we set $\ell=\ell(g)=4\log\log(g)$, then with high probability\footnote{A sequence of events $(A_g)_{g\geq 2}$ is said to hold \emph{with high probability as $g\to\infty$} if $\lim_{g\to\infty} \PP(A_g) = 1$.} as $g\to\infty$, the diameter of $S_{g,\ell}$ satisfies the bound from the theorem. We note that using the same method, the constant $C>0$ that appears in the theorem could also be made explicit. Because we have no reason to believe that the multiplicative constant in front of $\log\log(g)$ we find is optimal, we will not pursue this.

The remainder of this note consists of three sections. In Section \ref{sec_crit_exp} we present a bound on the critical exponent on a certain Fuchsian group $\Gamma_\ell$. This bound was first derived by McMullen \cite{McMullen}. For completeness, we present a proof. In Section \ref{sec_lattice_points}, we use a sub-multiplicativity argument to prove an effective estimate on the number of points at bounded distance from a fixed point in an orbit of $\Gamma_\ell$ on the hyperbolic plane $\HH^2$. In Section \ref{sec_proof}, we combine the two inputs and prove the main theorem.

\subsection*{Acknowlegdements}

We would like to thank Asma Hassannezhad for pointing out a useful reference. We also acknowledge funding from the grant ANR-23-CE40-0020-02 ``MOST''. J.M. is supported by  Austrian Science Fund (FWF) grant 10.55776/PAT1878824 on “Random Conformal Fields". B.P. would like to thank the Isaac Newton Institute for Mathematical Sciences, Cambridge, for support and hospitality during the programme Geometric spectral theory and applications, where work on this paper was undertaken. This work was supported by EPSRC grant no EP/Z000580/1.

\section{Critical exponents}\label{sec_crit_exp}

For $\ell>0$, we will let $H_\ell\subset\HH^2$ denote a right-angled hexagon with three non-consecutive sides of length $\ell/2$. We will moreover write $\Gamma_\ell$ for the group generated by the three reflections in the sides of length $\ell/2$ of $H_\ell$. 

Given a Fuchsian group $\Gamma < \mathrm{Isom}(\HH^2)$, we will denote the Hausdorff dimension of its limit set by $\delta(\Gamma)$. If $\Gamma$ is geometrically finite, this quantity is known to coincide with the critical exponent of the group \cite{Patterson}. In what follows, we will need an estimate on $\delta_\ell := \delta(\Gamma_\ell)$ as $\ell\to\infty$. As explained by McMullen in  \cite[Theorem 3.5]{McMullen}, such an estimate follows from the work of Dodziuk--Pignataro--Randol--Sullivan \cite{DodziukPignataroRandolSullivan}. Using work by Colbois \cite{Colbois} and Burger \cite{Burger} this bound could be made more precise. This however won't influence the multiplicative constant in front of $\log\log(g)$ in Theorem \ref{th:main} (but it does influence the constant $C$ in the same theorem), so we won't work this out here.

\begin{lem}\label{lem_crit_exp}
There exists a universal constant $C>0$ such that
\[
\abs{1-\delta_\ell} \leq C\cdot e^{-\ell/4}
\]
for all $\ell \geq 1$.
\end{lem}

\begin{proof} In the interest of completeness and also because McMullen describes the estimate as a function of a different parameter, we will provide a proof.

If we double $H_\ell$ along its three sides of length $\ell/2$, then we obtain a pair of pants. This pair of pants is the convex core of $\Lambda_\ell\backslash\HH^2$, where  $\Lambda_\ell<\Gamma_\ell$ is the index two subgroup of orientation preserving elements. Since the critical exponent does not change when we pass to a subgroup of finite index, we use $\Lambda_\ell$ instead. 

The Elstrodt--Sullivan formula \cite[Theorem 2.17]{Sullivan} states that 
\[
\delta(\Lambda_\ell)\cdot (1-\delta(\Lambda_\ell)) = \lambda_0(\Lambda_\ell\backslash\HH^2),
\]
where $\lambda_0(\Lambda_\ell\backslash\HH^2)$ denotes the bottom of the spectrum of the Laplacian on $L^2(\Lambda_\ell\backslash\HH^2)$.

Moreover, \cite[Theorem 1.1']{DodziukPignataroRandolSullivan} says that there is a universal constant $A>0$ such that, whenever $\lambda_0(\Lambda_\ell\backslash\HH^2)<\frac{1}{4}$,
\[
\frac{1}{A} \cdot L_0(\Lambda_\ell\backslash\HH^2) \quad \leq \quad \lambda_0(\Lambda_\ell\backslash\HH^2) \quad \leq \quad A \cdot  L_0(\Lambda_\ell\backslash\HH^2),
\]
where $L_0(\Lambda_\ell\backslash\HH^2)$ is the length of the shortest simple closed geodesic that disconnects $\Lambda_\ell\backslash\HH^2$. $\Lambda_\ell\backslash\HH^2$ has exactly three simple closed geodesics that all have the same length and that all disconnect $\Lambda_\ell\backslash\HH^2$. For instance using the continuity of Hausdorff dimension proved in \cite[Theorem 3.1]{McMullen} and the fact that in the limit $\ell\to\infty$, the group is a lattice, we obtain that the condition $\lambda_0(\Lambda_\ell\backslash\HH^2)<\frac{1}{4}$ holds when $\ell$ is large enough.

Now using the formulas from \cite[Page 454]{Buser_book}, one obtains that 
\[
L_0(\Lambda_\ell\backslash\HH^2) \stackrel{\ell\to\infty}{\sim} 4 \cdot e^{-\ell/4}.
\]
So putting this together with the previous two equations, we obtain the lemma.
\end{proof}

\section{Lattice point counting}\label{sec_lattice_points}

The second input we need is a bound on the difference between the number of points in an orbit of $\Gamma_\ell$ at a finite distance from a base point and the asymptotic prediction. The bound we need is a special case of \cite[Corollary 2.11]{Mathien}. 

We start by fixing some notation. We once and for all fix a base point $\mathbf{o}\in H_\ell$. In order make the computations below effective, we will let $\mathbf{o}$ be the unique point that is equidistant form the three sides of length $\ell/2$. Equivalently, this is the center of the largest inscribed disk. The distance from $\mathbf{o}$ to these three sides, that will show up in multiple bounds below, will be denoted $C_\ell$.

We will write $\calT_3=\Gamma_\ell\cdot \mathbf{o}$ for the orbit of this point, and we will call these points \textbf{lattice points}. The reason for the notation $\calT_3$ is that the set of lattice points is naturally endowed with a tree structure: two points are \textbf{adjacent} if the corresponding hexagons share an edge. Seen as a graph, this set is then isomorphic to the infinite trivalent tree. In what follows, we will call $\mathbf{o}$ the root of the tree $\mathcal T_3$, leading to an \textbf{ancestor/descendant} order for lattice points.

We will also write $B_R(\mathbf{o})\subset\HH^2$ for the disk of radius $R$ around $\mathbf{o}$. We define the discrete ball of radius $R$ around a lattice point $x$ as 
\[
\mathcal B_R(x) = \calT_3 \cap B_R(x)
\] 
and write
\[
N_\ell(R) = \#\calB_R(\mathbf{o})
\]
and note that in fact $N_\ell(R)$ is the cardinality of $\mathcal B_R(\gamma\cdot \mathbf{o})$ for any $\gamma \in \Gamma_\ell$. 

The main point of this section is the following estimate:

\begin{prp}\label{prop:speedCV} For all $\ell\geq 1$ and all $R>0$, we have
\[
\frac{\log(N_\ell(R))}{R} \geq \delta_\ell- \frac{\ell+8C_\ell+\log(20/3)}{R}.
\]
\end{prp}

For sake of clarity and completeness, we give again the details of Proposition \ref{prop:speedCV} here, using some simplifications due to the fact that the model we use is less random than the general case presented in \cite{Mathien}. We note that it's also possible to produce a two-sided bound (see \cite[Corollary 2.11]{Mathien}), but because we only need one side, we will not include details here.

The proof of Proposition \ref{prop:speedCV} can be reduced to the following lemma, based on the underlying tree structure. It expresses the fact that up to a multiplicative constant, $N_\ell(R)$ is sub-multiplicative, so $\log N_\ell(R)$ is larger than a certain linear function of $R$.

\begin{lem}\label{lem:mult}
	Fix $R,\ell > 0$.
	Then \[\forall r \geq 0,\quad  N_\ell(R+r) \leq \frac{20}{3} e^{8C_\ell + \ell}  N_\ell(R) N_\ell(r).\]
\end{lem}

\begin{proof}[Proof of Proposition \ref{prop:speedCV}]
	As a consequence of Lemma \ref{lem:mult}, the function $R \mapsto \frac{20}{3} e^{8C_\ell+\ell} N_\ell(R)$ is sub-multiplicative, and so by Fekete's lemma, we have that \[\delta_\ell = \lim_{R \to \infty} \frac{\log N_\ell(R)}{R} = \inf_{R \geq 0} \frac{\log N_\ell(R) + \log (20/3) + 8C_\ell +\ell}{R},\] so \[\frac{\log N_\ell(R)}{R}-\delta_\ell \geq - \frac{\log (20/3) + 8C_\ell +\ell}{R},\]
which proves the proposition.
\end{proof}

Before proving Lemma \ref{lem:mult}, let us introduce some additional notation, and some elementary facts.

\begin{defn}
For $R\geq 0$, we define \[\mathcal S_R = \lbrace y \in \calT_3;\; R-2C_\ell \leq \mathrm{d}(\mathbf{o}, y) < R\rbrace.\]
\end{defn}

The definition of $\mathcal S_R$ is motivated by the following fact:
\begin{prp}\label{prop:ancestor}
	Let $x$ be a lattice point such that $\mathrm{d}(\mathbf{o}, x) > R-2C_\ell$. Then there exists an ancestor $y$ of $x$ in the rooted tree $\mathcal T_3$ such that $y \in \mathcal S_R$.
\end{prp}

\begin{proof}
	Consider the set of ancestors $z$ of $x$ such that $\mathrm{d}(\mathbf{o}, z) > R-2C_\ell$. This is a non-empty set (it contains $x$ itself), and its element $y$ of smallest depth satisfies $\mathrm{d}(\mathbf{o}, y) <R$. Otherwise, its parent $p$ would also be at distance more than $R-2C_\ell$ from $\mathbf{o}$, because $\mathrm{d}(p, y) < 2C_\ell$.
\end{proof}
In addition, because of the exponential growth rate, the cardinality $N_R$ of $\mathcal B_R(o)$ is controlled by the one of $\mathcal S_R$.
\begin{prop}\label{prop:BorS}
	\[\# \mathcal  S_R \leq N_\ell(R) \leq 2\# \mathcal  S_R.\] 
\end{prop}

\begin{proof}
	It is for example a consequence of \cite[Proposition 2.14]{Mathien}: $\mathcal  B_R$ is a subtree of $\mathcal T_3$, and $\mathcal  S_R$ contains at least all the points of degree 1 (the leaves) and 2 (the nodes with one child) of this tree. Because the number of vertices of degree 3 in such a tree is bounded by the number of leaves, the results follows.
\end{proof}

We now have all what we need to prove Lemma \ref{lem:mult}.
\begin{proof}[Proof of Lemma \ref{lem:mult}]
	Take $R, r >0$. We bound $\# \mathcal S_{R+r}$ rather than $N_\ell(R+r)$, and then apply Proposition \ref{prop:BorS}.
	
	Let $x \in \mathcal S_{R+r}$. Because of Proposition \ref{prop:ancestor}, the geodesic from $\mathbf{o}$ to $x$ intersects a hexagon $H$ such that its center $x_H = \gamma_H \cdot \mathbf{o}$ is in $\mathcal S_R$. Let $l$ be the length of this geodesic, and $h$ be the first point of this geodesic in $H$ (see Figure \ref{fig:pathgeod}).

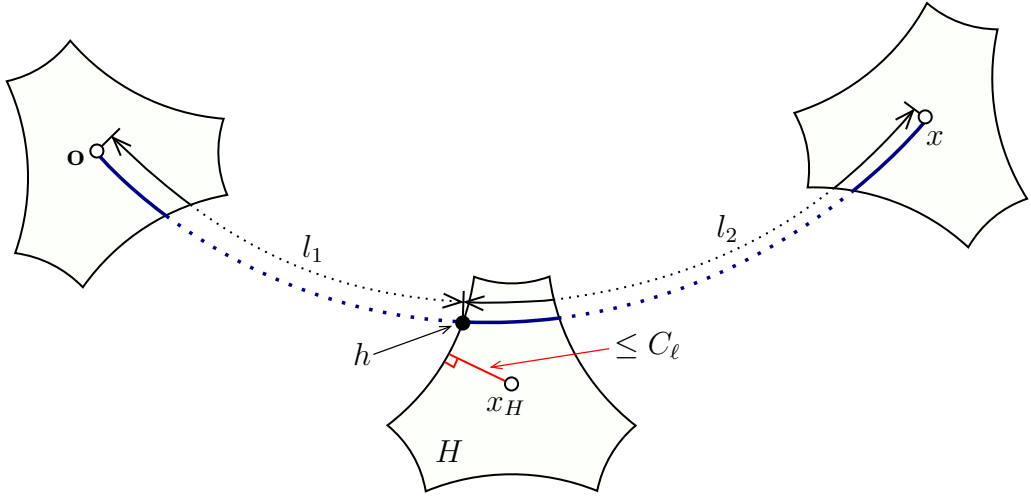
\begin{figure}[ht]
	\centering
	\begin{overpic}{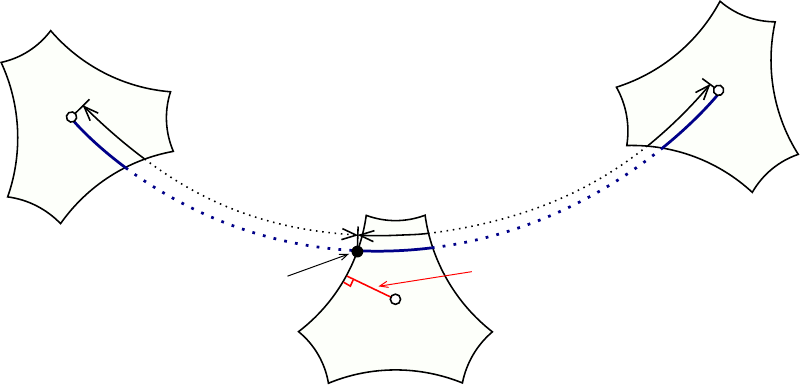}
	\put(6,32){$\mathbf{o}$}
	\put(47,8){$x_H$}
	\put(42,3){$H$}
	\put(34,12.5){$h$}
	\put(59.5,13.5){$\leq C_\ell$}
	\put(29,23){$l_1$}
	\put(69.5,25){$l_2$}
	\put(90,34){$x$}
	\end{overpic}
	\caption{A path between $\mathbf{o}$ and some other vertex $x$. All the symbols refer to the proof of Lemma \ref{lem:mult}.}
	\label{fig:pathgeod}
\end{figure}

	 The point $h$ is on one of the sides of length $\ell/2$ of $H$, so we have  \[\mathrm{d}(x_H, h) \leq C_\ell+\frac{\ell}{4}.\] The point $h$ divides the geodesic into two parts, one from $\mathbf{o}$ to $h$ of length $l_1$ and one from $h$ to $x$ of length $l_2$.
	Because $x_H$ is in $\mathcal S_R$, we have that \[R-2C_\ell \leq \mathrm{d}(\mathbf{o}, x_H) \leq \mathrm{d}(\mathbf{o}, h) + \mathrm{d}(h, x_H)\leq  l_1 + \frac{\ell}{4} + C_\ell.\] As a consequence, \[l_1 \geq R - 3C_\ell - \frac{\ell}{4},\] and so \[\l_2 = l - l_1 \leq r + 3C_\ell + \frac{\ell}{4}.\] 
In particular, \[\mathrm{d}(x, x_H) \leq l_2 + \frac{\ell}{4} + C_\ell = r +4C_\ell + \frac{\ell}{2}.\] 
	
We now define for $y\in \calT_3$,
\[\mathcal B_R^+(y) = \{\text{descendants }z \text{ of }y\text{ such that }\mathrm{d}(y, z)< R\}
\]	
and observe that by symmetry, for any lattice point $y \neq \mathbf{o}$, we have $\#\mathcal B_R^+(y) = \frac{2}{3} N_\ell(R)$. 	
	
	What we have just shown implies that $x \in \mathcal B_{r+4C_\ell+\ell/2}^+(x_H)$, and more generally, we have shown that
	\[ \mathcal S_{R+r} \subset \bigcup_{x_H \in \mathcal S_R} \mathcal B_{r+4C_\ell+\ell/2}^+(x_H).
	\] and as such

\begin{equation} \label{eq:ssmult} 
\#\mathcal S_{R+r} \leq \sum_{x_H \in \mathcal S_R} \#\mathcal B_{r+4C_\ell+\ell/2}^+(x_H).
	\end{equation}	
	
	In addition 
	\begin{equation} \label{eq:decendant_bound}
	\# \mathcal B_{r+4C_\ell+\ell/2}^+(x_H) = \frac{2}{3} N_\ell(r+4C_\ell+\ell/2) \leq \frac{4}{3} \# \mathcal S_{r+4C_\ell+\ell/2},
	\end{equation}
	 where we used Proposition \ref{prop:BorS} for the last inequality. One can exchange the role of $R$ and $r$ in equation $\eqref{eq:ssmult}$, and take $R = 4C_\ell + {\ell}/{2}$ to get that 
	\[
	\# \mathcal S_{r+4C_\ell+\ell/2} \leq \frac{2}{3}\; N_\ell(r)\cdot N_\ell(8C_\ell + \ell).
	\]
	On the other hand, for any $T >0$,  
	\[
	N_\ell(T) \leq 5 e^T.
	\] 
	Indeed, no matter what $\ell$ is, $H_\ell$ contains a disk of radius $\log(3)/2$ (the radius of the largest inscribed disk in an ideal triangle) around its midpoint. The area of this disk is $2\pi \cdot\left(\frac{2}{\sqrt{3}}-1\right)>\frac{3\pi}{10}$. So we get
	\begin{align*} \frac{3\pi}{10} N_\ell(T) &\leq \mathrm{Area}(B_{ T+\log(3)/2}(\mathbf{o})) \\ &\leq \frac{3\pi}{2} \cdot e^T.
	\end{align*}
	This means that $\# \mathcal S_{r+4C_\ell+\ell/2} \leq \frac{10}{3} N_\ell(r)\cdot e^{8 C_\ell + \ell}$. When we put this together with \eqref{eq:ssmult} and \eqref{eq:decendant_bound}, we obtain
	\[N_\ell(R+r) \leq 2\# \calS_{R+r} \leq  \frac{20}{3}\; e^{8C_\ell+\ell}\; \#\calS_R \cdot  N_\ell(r) \leq  \frac{20}{3}\; e^{8C_\ell+\ell}\; N_\ell(R) \cdot  N_\ell(r),
	\]
	where we have used Proposition \ref{prop:BorS} twice.
\end{proof}

\section{The proof of the main theorem}\label{sec_proof}

The estimate on the diameter of the random surfaces $S_{g,\ell}$ now follows along the same lines as in \cite{BollobasFernandezdelaVega,BudzinskiCurienPetri,Mathien}, using the improved input from the previous two sections. We will describe how this works here. We will focus on the new input, referring to the earlier papers whenever possible. 

Concretely, we will prove the following proposition, which immediately implies Theorem \ref{th:main} (potentially increasing the constant $C>0$ slightly).

\begin{prp} Set $\ell = \ell(g) = 4\log \log(g)$. There exists a universal constant $C>0$ such that
\[
\lim_{g\to\infty}
\PP\Big(\mathrm{diam}(S_{g,\ell})  \leq \log(g) + 25\cdot \log\log(g) + C\Big)  = 1.
\]
\end{prp}

\begin{proof}
We will start with some notation. $G_g$ will denote a random cubic graph on $V_g=\{1,\ldots,2g-2\}$, distributed according to the configuration model. So every vertex $v\in V_g$ corresponds to a copy of $P_\ell$ that we use to build $S_{g,\ell}$ and we will conflate the vertex $v$ and the midpoint of the corresponding pair of pants in what follows. Our goal is to show that, with high probability, for every $v,w\in V_g$ we can find a short (for the hyperbolic metric on $S_{g,\ell}$) path between the two corresponding pairs of pants.

We do this using a dynamical exploration (or peeling algorithm) of $S_{g,\ell}$ around these vertices. That is, given $v\in V_g$, we may define a sequence of random surfaces
\[
S_{g,\ell}^{(0)}(v),\quad S_{g,\ell}^{(1)}(v),\quad\ldots , \quad S_{g,\ell}^{(2g-2)}(v),
\]
where $S_{g,\ell}^{(0)}(v)$ is a disjoint union of $2g-2$ pairs of pants, labeled using the elements of $V_g$ and, for $t=0,\ldots,2g-3$, the surface $S_{g,\ell}^{(t+1)}(v)$ is obtained from $S_{g,\ell}^{(t)}(v)$ as follows
\begin{enumerate}
\item~\label{item_normal} if the component of $S_{g,\ell}^{(t)}(v)$ containing $v$ has a non-empty boundary, we pick the boundary component closest (for the hyperbolic metric) to $v$ and glue it to another boundary component of $S_{g,\ell}^{(t)}(v)$ that is picked uniformly at random,
\item~\label{item_disconnect} if not, we pick two boundary components of $S_{g,\ell}^{(t)}(v)$ uniformly at random and glue them together.
\end{enumerate}
We observe that the distribution of $S_{g,\ell}^{(2g-2)}(v)$ is the same as that of $S_{g,\ell}$. 

We will run the exploration until time $\tau_2 = (25(g-1)\log(g-1))^{1/2}$. We will also set $\tau_1=(g-1)^{1/2-\eps}$ and call the part of the exploration during the times in $[1,\tau_1]$ the \textbf{first phase} and the part during the times $[\tau_1+1,\tau_2]$ the \textbf{second phase}. Step $t+1$ will be called \textbf{bad} if during that step, two boundary components of the connected of $S_{g,\ell}^{(t)}(v)$ containing $v$ are paired. Using exactly the same argument as in \cite[page 372]{BudzinskiCurienPetri} and using that $x\mapsto \frac{x+1}{6g-6-2x-1}$ is an increasing function of $x \in [1,\tau_2]$ for all $g\geq 2$, we obtain that 
\[
\PP\left(\begin{array}{c}\text{at least } k \text{ bad steps are made} \\ \text{during the first phase}\end{array} \right) \leq \frac{1}{k!}(g-1)^{-2\eps k},
\]
for any $k \in \NN$ and any $\eps>0$. We will not fix these parameters just yet, but we will assume that they are such that $\frac{1}{k!}(g-1)^{-2\eps k}\stackrel{g\to\infty}{=} o\left((g-1)^{-2}\right)$. Likewise,
\[
\PP\left( \begin{array}{c}\text{at least } \log^3(g-1) \text{ bad steps are} \\ \text{made during the second phase}\end{array}\right) \stackrel{g\to\infty}{=} o\left((g-1)^{-2}\right).
\]
Now we need a slightly more effective version of \cite[Lemma 4]{BudzinskiCurienPetri} (see also \cite[Lemma 3.4]{Mathien} that applies in a more general context). We will set
\[
R_t = \max\left\{ \mathrm{d}(v,\beta); \; \begin{array}{c} \beta \text{ a boundary component of the connected} \\
\text{component of } S^{(t)}_{g,\ell}(v)\text{ containing }v\end{array}  \right\}.
\]
Our goal is to prove an upper bound on $R_{\tau_2}$ under the assumption that fewer than $k$ bad steps happen in the first phase and at most $\log^3(g-1)$ during the second. We also assume that we are never in the situation described in item \ref{item_disconnect} above (if not the surface would be disconnected).

We observe that $S_{g,\ell}$ is covered by a tree of pants: a hyperbolic surface $T_\ell$ of infinite area that is obtained by gluing countably many copies of $P_\ell$ together according to the combinatorics of an infinite trivalent tree. Alternatively, we obtain $T_\ell$ by doubling the convex subset $\Gamma_\ell \cdot H_\ell\subset \HH^2$ along its boundary. We will keep one of the copies of the orbit $\Gamma_\ell\cdot\mathbf{o}$ and will call the corresponding points on $T_\ell$ and $S_{g,\ell}$ the \textbf{midpoints} of the pairs of pants they lie in. Like in \cite[page 388]{BudzinskiCurienPetri}, we observe that the number of midpoints in $T_\ell$ at hyperbolic distance at most $R$ from a fixed midpoint equals $N_\ell(R)$.

We now start by analyzing the first phase. Our assumptions on the number of bad steps imply that $S^{(6k)}_{g,\ell}(v)$ has at least $2k$ boundary components. Thus among the descendants of at least one of them, that we will call $\eta$, no bad steps will happen during the first phase. Because of the order in which we discover the pairs of pants in $S_{g,\ell}$, all the descendants of $\eta$ that are at distance at most $R_{\tau_1}-3 \cdot C_\ell-\ell/4$ from $v$ are a part of $S^{(\tau_1)}_{g,\ell}(v)$. Now we need to figure out their distance from the midpoint $m_\eta$ of the pair of pants ``before'' $\eta$. Let $x$ be one of these midpoints. We have
\[
d(x,m_{\eta}) \leq d(x,v) + d(v,m_\eta)
\]
So, because $d(v,m_\eta) \leq R_{6k}+\ell/4+C_\ell$, we obtain that all descendants of $m_\eta$ at distance at most $R_{\tau_1}-R_{6k}-\ell/2-4C_\ell$ from $v$ are part of $S^{(\tau_1)}_{g,\ell}(v)$. So we obtain that

\[
\frac{2}{3}N_\ell(R_{\tau_1}-R_{6k}-\ell/2-4C_\ell) \leq \tau_1.
\]

If $S^{(\tau_1)}_{g,\ell}$ is of genus $0$ (i.e. if no bad steps happen at all up until time $\tau_1$), then it would have $\tau_1+3$ boundary components. Because there are at most $k$ bad steps, $S^{(\tau_1)}_{g,\ell}$ has at least $\tau_1+3-3k$ boundary components. Because there are at most $\log^3(g)$ bad steps during the second phase, at least $\tau_1+3-3k-2\log^3(g)$ of them have no bad steps among their descendants. With the same argument as before, this means that
\[
(\tau_1-3k-2\log^3(g))\cdot \frac{2}{3}N_\ell(R_{\tau_2}-R_{\tau_1}-\ell/2-4C_\ell) \leq \tau_2.
\]
Putting this together with the lower bound on $\tau_1$, we get
\begin{multline*}
\log\Big(\frac{2}{3}N_\ell(R_{\tau_1}-R_{6k}-\ell/2-4C_\ell) -3k-2\log^3(g)\Big)  \\
+ \log\Big( \frac{2}{3}N_\ell(R_{\tau_2}-R_{\tau_1}-\ell/2-4C_\ell) \Big)
\leq \log(\tau_2).
\end{multline*}
We will now start writing ``$C$'' to indicate a constant that might change from line to line. Moreover, we will assume that $k$ and $g$ are such that $3k+2\log^3(g) \leq \frac{2}{6}N_\ell(R_{\tau_1}-R_{6k}-\ell/2-4C_\ell)$ and use that, when $y \in (0,x/2]$, $\log(1-y/x) \geq -2y/x$. This yields that
\[
\log\Big(N_\ell(R_{\tau_1}-R_{6k}-\ell/2-4C_\ell) \Big) 
+ \log\Big( N_\ell(R_{\tau_2}-R_{\tau_1}-\ell/2-4C_\ell) \Big) 
\leq \log(\tau_2)+C.
\]
Now using Proposition \ref{prop:speedCV}, we obtain that
\[
\delta_\ell\cdot (R_{\tau_1}-R_{6k}-\ell/2-4C_\ell) + \delta_\ell\cdot (R_{\tau_2}-R_{\tau_1}-\ell/2-4C_\ell) \leq \log(\tau_2)+ 2\ell+16C_\ell+C
\]
Using  \cite[Example 2.2.7.(iii)-(v)]{Buser_book}, it can also be shown that $C_\ell \leq C\left(1 + \frac{1}{\ell^2}\right)$. So we apply this in combination with Lemma \ref{lem_crit_exp} and the fact that $R_{6k} \leq (12k+1)\cdot C_\ell$, to obtain
\[
R_{\tau_2}  \leq \log(\tau_2) + C e^{-\ell/4} \log(\tau_2) +  3 \ell + C\cdot k + C 
\]
From the above, we get that our best choice is to set $\ell=4\log\log(g)$ and to make $k$ as small as possible a constant. This means that $k=3$, because we need there to exist $\eps<\frac{1}{2}$ such that $2\eps k > 2$. This also guarantees that $3k+\log^3(g) \leq \frac{2}{6}N_\ell(R_{\tau_1}-R_{6k}-\ell/2 - 4C_\ell)$ for large enough $g$. So we obtain that
\[
R_{\tau_2}  \leq \frac{1}{2}\log(g) + \frac{25}{2} \log\log(g) + C.
\]
Using the exact same proof as in \cite[page 374]{BudzinskiCurienPetri}, one proves that, with probability $1-o\left((g-1)^{-2}\right)$, the neighborhoods we explore around $v$ and $w$ merge for any $v,w\in V_g$. Summing over the $\leq 4g^2$ pairs vertices and using that the surface doesn't disconnect with high probability, we conclude that the diameter is at most $2\cdot R_{\tau_2}$, which proves the proposition.
\end{proof}

\bibliography{bib}
\bibliographystyle{alpha}

\end{document}